\documentclass[12pt]{article}
\usepackage{geometry}                
\usepackage{yhmath} 
\usepackage{graphicx}
\usepackage{amssymb}
\usepackage{amsmath}
\usepackage{amsthm}
\usepackage{amsfonts}
\usepackage{color}
\usepackage{amsmath,amsfonts, amssymb, amsthm, amscd,graphicx}
\usepackage{latexsym}
\usepackage{epsfig,epsf,psfrag}
\usepackage{nonfloat}
\newcommand{\R}{\ensuremath{\mathbb{R}}}

\newcommand{\Z}{\ensuremath{\mathbb{Z}}}
\newcommand{\N}{\ensuremath{\mathbb{N}}}

\newcommand{\E}{\ensuremath{\mathbb{E}}}

\newcommand\EE{{\mathbb E}}
\newcommand\PP{{\mathbb P}}
\newcommand{\Ai}{\ensuremath{\mbox{Ai}}}

 \newcommand{\eps}{\varepsilon}
 \newtheorem{theorem}{Theorem}
 \newtheorem{lemma}[theorem]{Lemma}
 \newtheorem{propo}[theorem]{Proposition}

 \usepackage{eucal}

\title{A clever (self-repelling) burglar}
\author {Laure Dumaz\footnote {Ecole Normale Sup\'erieure, Universit\'e Paris-Sud and TU Budapest}}
\begin{document}

\maketitle

\begin{abstract}
We derive the following property of the ``true self-repelling motion'', a continuous real-valued self-interacting process $(X_t, t \ge 0)$
introduced by B\'alint T\'oth and Wendelin Werner. Conditionally on its occupation time measure at time one (which is the information about how much time it has spent where before time one), the law of $X_1$ is uniform in a certain admissible interval. This interval can be much shorter than the interval of its visited points but it has a positive probability (that we compute) to be this whole set. All this contrasts with the corresponding conditional distribution for Brownian motion that had been studied by Warren and Yor.
\end{abstract}

\section{Introduction}

The true self-repelling motion (TSRM) is a continuous real-valued process $(X_t, t \ge 0)$ constructed by B\'alint T\'oth and Wendelin Werner in \cite {TothWerner},  that is locally
self-interacting with its past occupation-time measure. 
It can be understood as the scaling limit of certain discrete self-repelling integer-valued random walks. 

One of the key-features of TSRM, that in fact enables its construction, is that almost surely, at any given time $t \ge 0$, its occupation time measure $\mu_t$ on $\R$ defined by 
$$ \mu_t ( I ) = \int_0^t 1_{X_s \in I } ds $$
for all interval $I$, has a continuous density $\Lambda_t (x)$ with respect to the Lebesgue measure:
$$\mu_t (I) = \int_I \Lambda_t (x)\, dx.$$
In other words, if the walker $X$ walks with a bag that continuously looses sand, then after time $t$, the sand profile (given by $x \mapsto \Lambda_t (x)$) is a continuous function. Recall that such a property is true for Brownian motion, but that it fails to be true for smooth evolutions $t \mapsto X_t$ (as it would create a discontinuity of the sand profile -- also referred to as the local time profile -- at the point $X_t$). 

The process $(X_t, t \ge 0)$ is interesting because its paths are of a very different type than Brownian motion (they do not have a finite quadratic variation for instance). It will turn out to be relevant for the present paper
to note that $X$ does almost surely have times of local increase: There almost surely exist (many) positive times $s$ such that for some positive $\epsilon$, one has $X_{s-v} < X_s < X_{s+v}$ for all $v \le \epsilon$. Recall that this is almost surely never the case for Brownian motion (see \cite {Burdzy} and the references therein). 
Hence, it follows easily that with positive probability, there will exist exceptional times $s \in (0, 1)$ such that 
$$ X_v < X_s < X_w $$ 
for all $0 \le  v < s < w \le 1$. By symmetry, because $X$ and $-X$ are identically distributed, the same holds for time of decrease.

The position $X_s$ corresponding to such times $s$ can be detected by looking at the local time profile at time $1$. Indeed, they are exactly those points in the support $S$ of the local time profile, for which $\Lambda_1 (x) = 0$. Suppose that the local time profile $\Lambda_1$ is given and contains such points. They are exceptional (their Lebesgue measure is null) and a.s. the process cannot go fast two times through the same point (it is not possible to have $\Lambda_1(x) = 0$ for points $x$ visited more than once by the process before time $1$). Therefore, $X_0=0$ and $X_1$ have to be on different sides of each of these points $x$. Hence, it follows that necessarily, the position $X_1$ is located in the subinterval $I$ of $S$ that is separated from $0$ by all these cut-points. When there are no cut-points (and it happens with a positive probability we will compute in Section \ref{sec:probawholeint}), we define $I$ to be the entire support $S$. The observation of $\Lambda_1 (\cdot)$ therefore limits the possible values of $X_1$ to $I$.
The goal of the present paper is to prove that:
\begin {theorem}
\label {mainthm}
The conditional distribution of $X_1$ given $\Lambda_1( \cdot)$ is the uniform distribution in the interval $I$.
\end {theorem}
This may at first seem quite surprising. Indeed, it for instance means that the conditional distribution depends only on $I$ and not on the actual local time profile in $I$. 
On the other hand, we will see that this is a rather natural feature of TSRM, inherited from properties of a family of coalescing Brownian motions. 
 
\begin{figure}
 \centering
\includegraphics[width = 15cm]{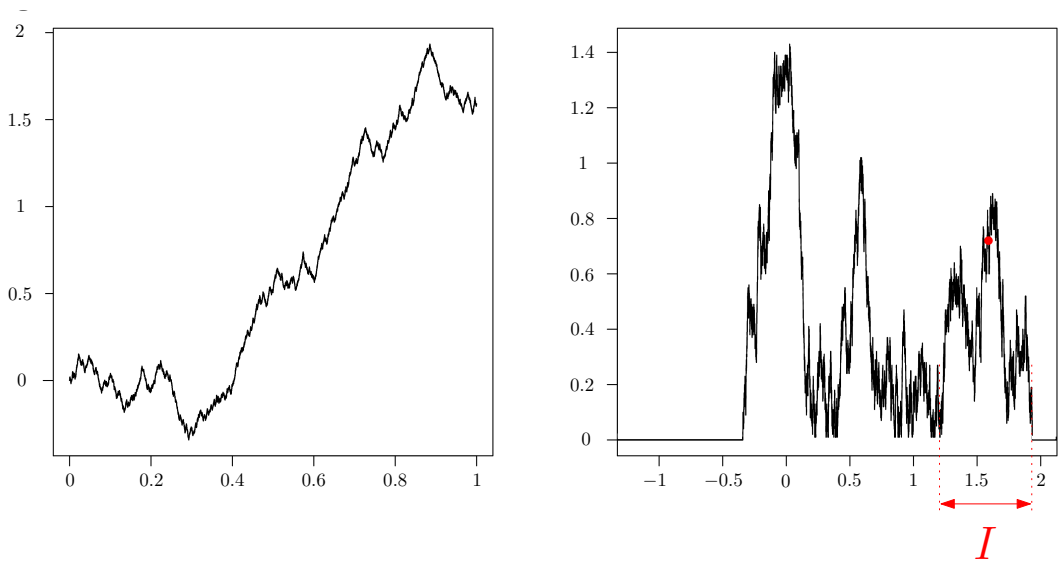}
\caption{Sample of the TSRM trajectory until $t=1$ and of its local time profile at time $1$.}
\end{figure}

This type of question was already studied in the case of Brownian motion by Warren and Yor in \cite{BBYor} and later also by Aldous in \cite{BBAldous}. The resulting distribution (of the position of the Brownian motion given its local time) was called ``Brownian burglar'' because one can imagine that someone is  trying to find a burglar moving like a Brownian motion and that the only piece of information one knows is the places he has previously robbed (or how many hotel bills he has paid etc.). We can keep here the same picture in mind except this self-repelling burglar is somehow more clever, because he manages to leave very little information behind. It is in fact a natural question to ask whether it is possible to find processes of a different kind with a similar property.

The TSRM is directly defined in the continuous setting without reference to a discrete model. It is in fact not so easy to prove that discrete self-repelling walks converge to the TSRM in strong topologies (see Newman and Ravinshankar \cite{NewmanCV}). But some of the properties related to TSRM are rather tricky to derive directly in the continuous setting, while their discrete counterparts are easy. 
A natural route to deriving them is therefore to control this property in the scaling limit (when the discrete model tends to TSRM); See for instance \cite {STW} for a use of such invariance principles for the derivation of the 
joint law of local times measures at different stopping times. This is also the approach that we will use in the present paper.

However, as we will see, some minor complications pop in due to the fact that the result corresponding to Theorem \ref {mainthm} in the discrete setting fails to be exactly true (it will hold only up to an error term that vanishes in the scaling limit). It will therefore be convenient to randomize time instead of considering the fixed time $1$, and we will in fact establish the following variant of Theorem \ref {mainthm}:
\begin {theorem}
 \label {main2}
Suppose that $\tau$ is an exponential random variable with mean $1$ that is independent of the TSRM $X$. Then, 
the conditional distribution of $X_\tau$ given $\Lambda_\tau( \cdot)$ is the uniform distribution in the interval $I_\tau$.
\end {theorem}
Here $I_t$ is the obvious generalization at time $t$ of $I$. The scaling property of TSRM i.e., the fact that for any positive $A$, 
$ (X_{A t}, t \ge 0) $ has the same law as $(A^{2/3} X_t, t \ge 0)$ together with the fact that $\tau$ can be read off from $\Lambda_\tau (\cdot)$ (it is the area under this curve) 
shows immediately that this statement is equivalent to Theorem \ref {mainthm}. 

\section {The result in the discrete setting}

Let us now describe a discrete self-repelling random walk $(\tilde X_n, n \ge 0)$ on the integers
introduced in \S 11 of \cite{TothWerner}, that we will use in the present paper, and establish the discrete analog of Theorem \ref {main2} for this walk. Note that  other self-repelling walks do also converge to the TSRM (e.g. the Amit-Parisi-Peliti true self-avoiding walk \cite {APP}) but this one turns out to be very convenient for our purposes because its ``local times'' form a simple discrete web.
It can be defined in two equivalent ways that we now briefly review.

\paragraph{Self-interacting random walk definition.} The first approach is to view $(\tilde X_n, n \ge 0)$ as a self-interacting random walk. Throughout the paper, we will view the set $E=\Z + 1/2$ as the set of edges that join two consecutive integers. 
 For all $n \in \N$ and $e \in E$, let $l(n, e)$ denotes the number of jumps along the edge $e$ before the $n$-th step:
$$
l(n,e) = \#\left\{k \in \{0,\cdots,n-1\},\; \{\tilde X_k,\tilde X_{k+1}\} = \{e-1/2,e+1/2\}\right\}.
$$
In fact, it is convenient to define a slight modification $\ell_n (e)$
of $l(n, e)$ by 
$$
\ell_n (e) = l(n,e) + a(e)
$$
where $a(e)$ is equal to  $0$ or $-1$ depending on whether $|e| - 1/2$ is even or odd, respectively.

The law of the random walk $(\tilde X_n)$ is then defined inductively as follows: $\tilde X_0=0$ and for all $n \ge 0$, if we define 
$$ \ell_n^-:=  \ell_n (\tilde X_n -1/2)  \hbox { and } \ell_n^+ := \ell_ n (\tilde X_n +1/2) $$
the slightly modified local times on the edges neighboring $\tilde X_n$, then 
\begin{align*}
 \PP\left(\tilde X_{n+1} = \tilde X_n + 1 | \tilde X_0,\cdots,\tilde X_n\right) &= 1 - \PP\left(\tilde X_{n+1} = \tilde X_n - 1 | \tilde X_0,\cdots,\tilde X_n\right)  \\
&= 
\left\{
\begin{array}{lll}
1 &\mbox{ if } &\ell_n^- > \ell_n^+\\
1/2 &\mbox{ if } &\ell_n^- = \ell_n^+ \\
0 &\mbox{ if } &\ell_n^- < \ell_n^+ \\ 
\end{array}\right.
\end{align*}
In other words, at step $n$, the walk chooses to jump along the neighboring edge it has visited less often in the past (modulo the initializing term $a$), and in case $\ell_n^+ = \ell_n^-$, it tosses a fair coin to choose its direction. 

It is important to remark that the initial state $a$ is chosen in such a way that one has $|a(e) - a(e+1)| = 1$ for all $e \in E \setminus \{\tilde X_0-1/2\} (= E \setminus \{-1/2\})$ and $a(-1/2) = a(1/2)$, and this rule perpetuates: $|\ell(n,e) - \ell(n,e+1)| = 1$ for all $n \in \N$ and all $e \in E$ except one, which is the edge $e = \tilde X_n -1/2$. At this point, we have  $\ell(n,e) - \ell(n,e+1) \in \{-2,0,2\}$. Thus, with such an initial condition, $\tilde X$ can be read off from $\ell$ (therefore, $\ell$ is Markov).
The initial condition $a$ is the ``flattest'' condition one can define which follows those rules. The condition on the initialization permits to avoid some artificial deterministic behaviors such as the one given by the initial local time null everywhere (in this case the walk would go deterministically in the direction of its first choice). Note that one can consider other natural initializations $a$ (such as the i.i.d. case, when $(a(k +1/2), k\in\N)$ and $(a(-k -1/2), k\in\N)$ are random and follow independent simple random walks starting at $0$), that will converge to TSRM with other initial condition (the i.i.d case converges towards the ``stationary'' TSRM defined in \S 10 of \cite{TothWerner}).

\paragraph{Second approach.} It turns out (and this is very simple to check, see \S 11 of \cite {TothWerner}) that this walk can also be interpreted in terms of a path that walks through a simple two-dimensional labyrinth in the 
upper half-plane. Let us write $\N^\# := \N \cup \{ -1 \}$ and let $F$ and $B$ be the lattices: 
\begin{align*}
 &F:= \{ (x-1/2 , h) \in (\Z - 1/2) \times \N^\#  \ : \  x+h \hbox { is odd}\}, \\
&B:= ((\Z-1/2) \times \N^\# ) \setminus F.
\end{align*}
We divide the upper half-plane into $1 \times 2$ rectangles of the type 
$(x-1/2, x+1/2) \times (h-1, h+1)$ for $(x-1/2,h) \in F$. In each rectangle associated with $(x-1/2,h) \in F$ such that $h \notin \{-1,0\}$, we toss independently a fair coin in order to choose between the two fillings depicted in Figure \ref {rectangleposs}: Either we draw two upwards parallel lines from $(x-1/2, h) \in F$ to $(x+1/2, h+1) \in F$ and from $(x-1/2, h-1) \in B$ to $(x+1/2, h) \in B$ or we draw two downwards lines from $(x-1/2, h) \in F$ to $(x+1/2, h-1) \in F$ and from $(x-1/2, h+1) \in B$ to $(x+1/2, h) \in B$.

\begin{figure}[!h]
 \centering
\includegraphics[width = 10cm]{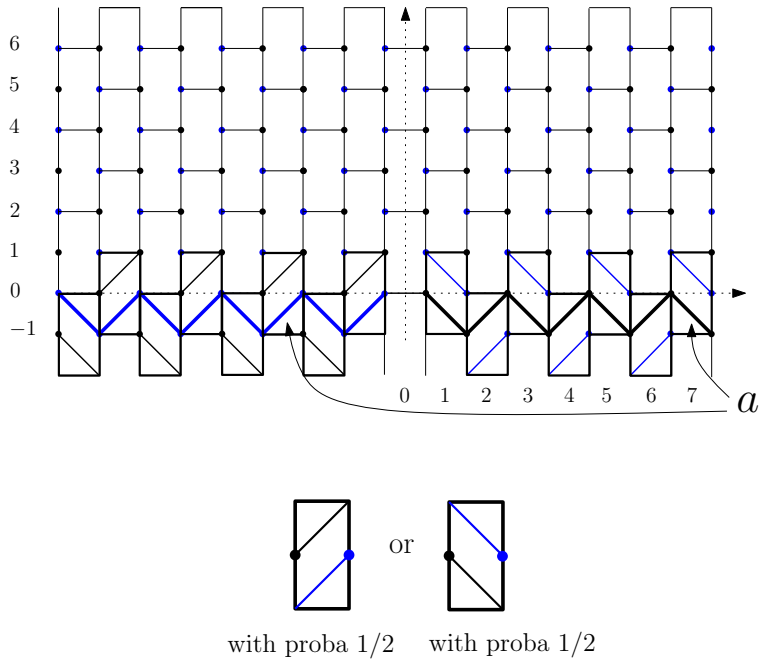}
\caption{The lattice with the initial conditions and the possible configurations in a rectangle}
\label{rectangleposs}
\end{figure}

If $h \in \{-1,0\}$, the lines are determined by the initial condition $a$ defined above: For all $e \in E \setminus \{-1/2\}$, we draw the line from $(e,a(e))$ to $(e+1,a(e+1))$ and its parallel line located in the same rectangle (see Figure \ref {rectangleposs} where the initial lines are drawn). When $e=-1/2$ we do not draw any line.

Note that in this way, the lines going through the lattice $F$ are a family of independent coalescing simple random walks going forward (i.e. to the right) starting from each point of $F$ and reflected above $0$ the left side of the origin and absorbed by $0$ on the right side of the origin. Similarly, the lines going through $B$ creates coalescing backwards random walks, that do never cross the forward lines (see Fig. \ref{sampleCRW}). Those families correspond to the discrete counterpart of the ``Brownian web'', a family of coalescing independent Brownian motions (see \S \ref{sec:CV} of this paper and \cite{TothWerner} for more details) and using this analogy, we call it the discrete web.
 
\begin{figure}[!h]
 \centering
\includegraphics[width = 9cm]{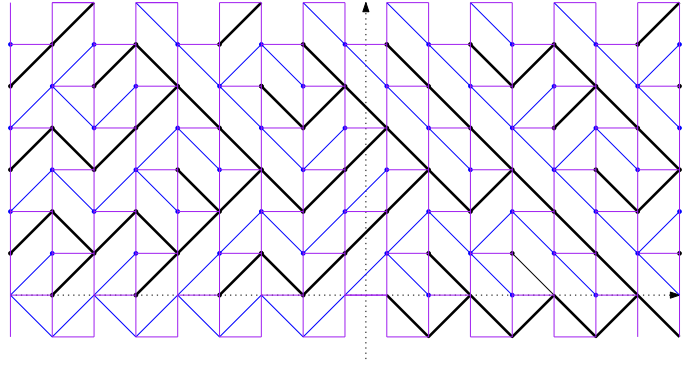}
\caption{A sample of the coalescing random walks}\label{sampleCRW}
\end{figure}

The two families of forward and backward lines create a random maze, with
one single connected component (this is a simple consequence of the
coalescing property). The path starting at $(0,0)$ that explores this maze (see Fig. \ref{RWWeb}) can be viewed as a discrete path we denote $(\tilde X_n,\tilde H_n) \in \Z \times \N$.
As shown in \cite{TothWerner}, its first coordinate has the same distribution as the $\Z$-valued random walk defined above
and its second coordinate corresponds to the average of the slightly modified local times at time $n$ on the two edges adjacent to $\tilde X_n$, i.e., 
$\tilde H_n =  (\ell_n^+ + \ell_n^-) /2$.

\begin{figure}
\includegraphics[width = 10cm]{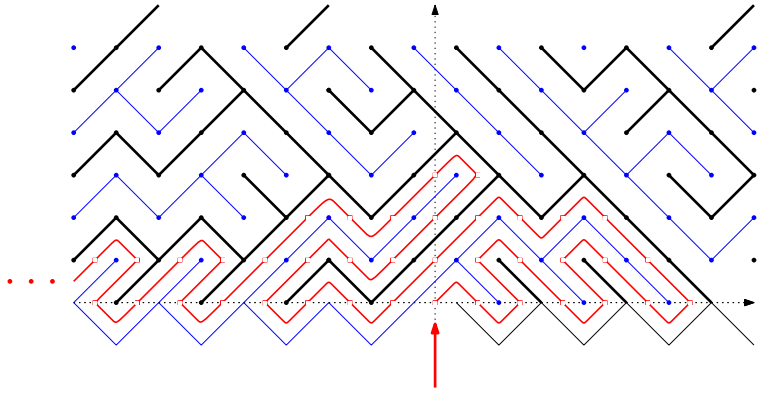}
\caption{Sample of the 39 first steps of $(\tilde X_n, \tilde H_n)$} 
\label{RWWeb}
\end{figure}
 
\medskip

For each $(e_0,h) \in F$ with $h \ge 1$, and each $e \in E$ with $e \ge e_0$, we denote by 
$\tilde  \Lambda_{e_0,h}(e)$ the value at $e$ of the \emph{forward} line in the web that starts at $(e_0,h)$ (it is a simple random walk indexed by $E$ and absorbed by $0$)
When $(e,h) \in B$, we use the same notation $\tilde \Lambda_{e,h}(e')$ to define the backward line starting in the left direction from $(e,h)$ (defined for $e'  \le e$). Then, at time $n$, if the walker is at the position $x$,
the local time profile $e \mapsto \ell_n (e)$ is equal to the lines $\tilde \Lambda_{x+1/2, \ell_n^+} ( e)$ and $\tilde \Lambda_{x-1/2, \ell_n^-} (e)$ respectively for $e >x$ and $e< x$. Each time the walk is at the bottom line of a rectangle, it discovers the status of the rectangle (meaning that it reveals if the lines are upwards or downwards lines in the rectangle) and this corresponds to moments for which $\ell_n^+ = \ell_n^-$. The position at time $n+1$ will be $\tilde X_n \pm 1$ if it is upwards/downwards lines, respectively. When the walk is in the middle of a (revealed) rectangle i.e. when $\ell_n^+ \neq \ell_n^-$, it follows the direction given by the lines in the rectangle: It goes to the right/left is the lines are downwards/upwards lines, respectively.

\begin{figure}
 \centering\includegraphics[width = 12cm]{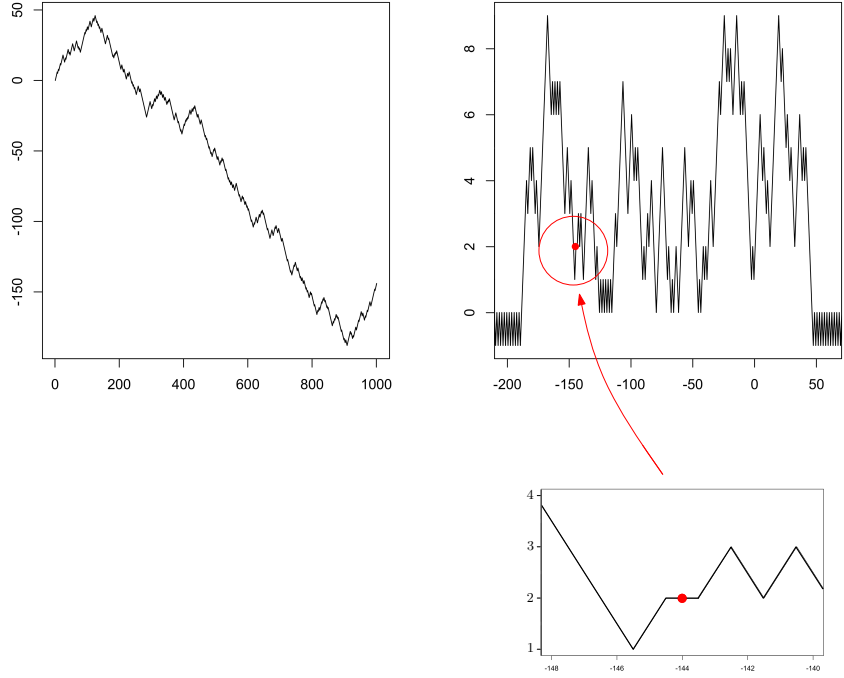}
\caption{Sample of the discrete model until time $n=1000$. On the left, $k \mapsto \tilde X_k$; on the right, the local time at time $n$ and $(\tilde X_n,\tilde H_n)$ dotted. A zoom is made around the position at time $n$}\label{figaveczoom}
\end{figure}

\paragraph{Modified local time.} Suppose now that a time $n$ is given, and that we observe the discrete local time profile at time $n$ i.e., the function $e \mapsto \ell_n (e)$. What can one say about the conditional law of $\tilde X_n$? 
A first observation is that one can immediately recover the location of $\tilde X_n$ by just looking at the local time profile, because as already noticed in the first definition, it is the only integer $x$ such that $|\ell_n(x+1/2) - \ell_n(x-1/2)| \neq 1$ (see the zoom in Fig. \ref{figaveczoom}).

So, in order to ``erase'' this information, it is natural to slightly modify $\ell_n$
 locally in the neighborhood of $\tilde X_n$. There are in fact several ways to proceed.
 The one that we choose to work with in the present paper is to simply concatenate the part of $\ell$ to the left of $x-1/2$ directly to the part to the right of $x+1/2$. However, then one may still be able to detect where such a surgery took place if $|\ell_n^- - \ell_n^+ | =2$. To avoid this problem, we will restrict our observations to the times at which 
$\ell_n^- = \ell_n^+$. 

\medbreak

We therefore denote by $(N(k), k \in \N)$ the times at which one really tosses a coin:  $N(0) = 0$, and for every $k \ge 0$, $N(k+1) := \inf\{n > N(k)\,:\, \ell_n^+ = \ell_n^-\}$. 
As one can somehow expect (half of the space-time points correspond to times $N(k)$ while the other half not, see Fig. \ref{RWWebcercles}), $N(k)$ almost surely behave like $2k$ up to lower order terms when $k \to \infty$ (it will be proved later, see \eqref{ineqpourk}). In fact, we have an exact formula which links $N(k)$ to $k$.
\begin{lemma}
 For all $k \ge 0$,
\begin {equation}
\label{N(k)}
 k = \frac{N(k) + \tilde H_{N(k)} }{2}.
 \end {equation}
\end{lemma}

There are various possible proofs of this simple combinatorial identity. We give a short one that uses our ``random walk'' interpretation:

\begin{proof}
 Indeed, the identity clearly holds for $k=0$ because  $N(0)=\tilde H_0=0$. Suppose it holds for $k$. Then, if $N(k+1)=N(k)+1$ it means that $\tilde H_{N(k+1)} = 1+ \tilde H_{N(k)}$, and therefore the identity for $k+1$ follows from that for $k$. If $N(k+1)= N(k) + j$ with $j \ge 2$, then it means that the TSRM has made $j-1$ forced moves, i.e. $(\tilde X_n, \tilde H_n)$ did ``slide down'' a slope. Therefore $\tilde H_{N(k+1)} = \tilde H_{N(k)} - (j -1) +1 = \tilde H_{N(k)} -j +2$ from which (\ref {N(k)}) follows. 
\begin{figure}[!h]
 \centering
\includegraphics[width = 7cm]{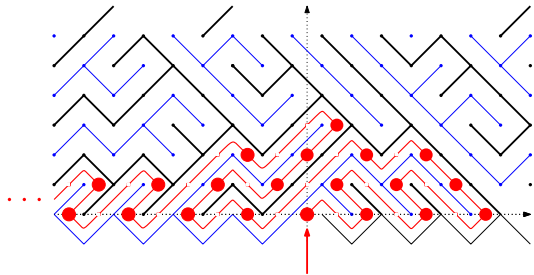}
\caption{Dots correspond to the times $N(k)$}
\label{RWWebcercles}
\end{figure}
\end{proof}

We will now mainly restrict ourselves to the set of times at which $\ell_n^+ = \ell^-_n$ i.e. to the case where $n \in N(\N)$ and we construct the curve $x \mapsto \tilde \ell_{n} (x)$ for all integer $x$ as follows: 
$$\tilde \ell_{n}(x) := \ell_n(x-1/2)  1_{x \le \tilde X_n} + \ell_n (x+1/2) 1_{x > \tilde X_n}.$$
In plain words, we shift the graph of $\ell_{n}$ horizontally by $1/2$ in the direction of $\tilde X_{n}$ on both sides of $\tilde X_n$. Note that for large enough $|x|$, the function $\tilde \ell_n$ coincides with the shifted function initial line $\tilde \ell_0$ i.e. with  $\tilde a(x)=  -1_{x \in (2 \Z + 1)}$. 

To sum up the previous few paragraphs: We have defined the function $x \mapsto \tilde \ell_{n}(x)$ which is a slightly modified local time, 
where we lost some information about the position of $\tilde X_{n}$ in the case where $\ell_n^+ = \ell_n^-$. In the rest of the paper, we shall refer to this modified curve $\tilde \ell$ as the ``modified local time''.

\begin{figure}
 \includegraphics[width=8cm]{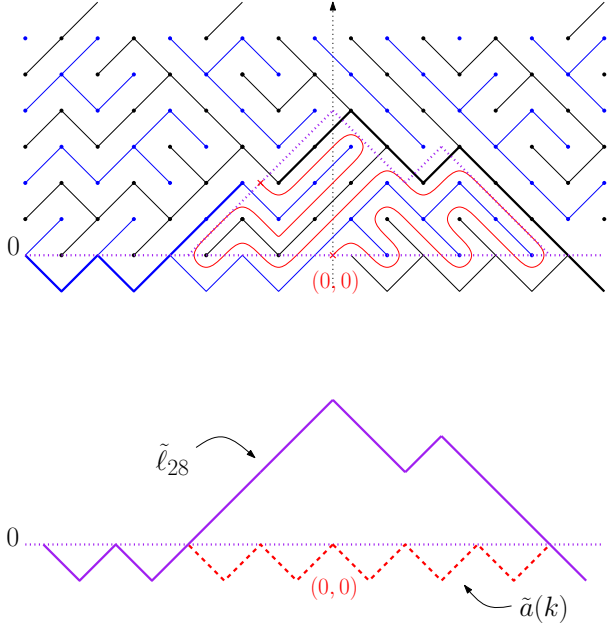}
\caption{The self-repelling walk until $n=28$ and its modified local time}
\label{modifiedloctime}
\end{figure}

\paragraph{Properties of modified local times.} In the remaining of this section, $f: \Z \mapsto \N$ will always denote a function that has a positive probability of being realized by some $\tilde \ell_{N(k)}$ for some $k \ge 0$, i.e. such that 
$$ \PP ( \exists k \ : \  f = \tilde \ell_{N(k)} ) > 0 .$$
We denote by $\mathcal C$ this set of functions.
Note that being in ${\mathcal C}$ implies certain necessary conditions for $f$: It is a function $f : \Z \to \N^\#$ such that $|f(x+1) - f(x)|= 1$ for all $x$, and 
if we define $m_- = m_- (f) := 1 + \max \{ x \in - \N  \, : \, f(x) = -1 \}$ and $m_+ = m_+ (f) := - 1 + \min \{ x \in \N \,: \, f(x) =-1 \}$, then 
$f= \tilde a$ on $(- \infty , m_-] \cup [m_+, \infty)$. Furthermore, we can note that $f(0)$ is necessarily even (this is just because if $\tilde X_n \le 0$, then $\tilde X$ has jumped an even number of times along the edge $1/2$, and if $\tilde X_n >0$, then it has jumped an even number of times on $-1/2$).

Note that it can happen that $f(x)$ is equal to $0$ for some integer values of $x$ in $(m_-, m_+)$ (but it can never be equal to $-1$ on this interval). 
Let $O(f)$ denote the number of zeroes of $f$ in this open interval. We define an excursion-interval to be a maximal interval on which $f$ is positive. Then $(m_-, m_+)$ can be split into $O(f)+1$ excursion-intervals. 

Note that $0$ necessarily belongs to the left-most or to the right-most excursion interval: 
Indeed, suppose for instance that $f= \tilde \ell_N$ and $\tilde X_{N} \ge 0$  (the case $\tilde X_{N} \le 0$ can be treated similarly), then every 
edge to the left of the origin 
is visited an even number of times (because $\tilde X$ has to jump equally often in both directions along that edge, as it starts to its right and ends up to its right), from which it follows that 
$0$ is in the left-most excursion interval of $f$.
Exactly the same argument shows that $\tilde X_N$ also belongs to the left-most or to the right-most interval which is at the opposite side of the $0$-interval, and also that $\tilde X_N$ cannot be one of the $O(f)$ internal zeros of $f$.
 
Hence, let us define 
 $I(f)$  to be equal to
         $[m_-, m_+]$  if $O(f) = 0$, and if $O(f) \ge 1$, then 
$$I (f) = (o_+, m_+] \hbox { or } [m_-, o_-)$$
depending on whether $ o_+ := \max \{ x < m_+: \ f(x) = 0 \}$ is positive or
 $ o_- : = \min \{ x > m_-: \ f(x) = 0 \} $  is negative. In the special case where $f = \tilde a$, we set $I(f)= \{0\}$.
 
Then necessarily, if $\tilde \ell_{N(k)} = f$ for some $k$, then this implies that $\tilde X_{N(k)} \in I(f)$.
\begin{figure}[!h]
 \centering
\includegraphics[width = 11cm]{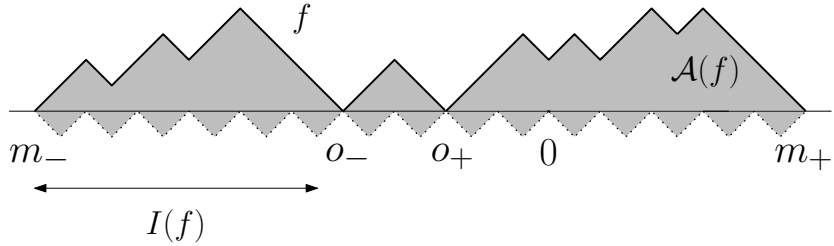}
\caption{A modified local time and some notations}\label{fig:defpourloctime}
\end{figure}
\medskip

\paragraph{Behavior of the walk given a modified local time curve $f$ and a position $x$.} Conversely, if $f \in {\mathcal C}$ is given and if $x \in I (f)$, we denote by $E_{x,f}$ the event that $f$ is a modified local time for which the surgery took place at position $x$ i.e. that for some $k$, 
$f= \tilde \ell_{N(k)}$ and $\tilde X_{N(k)} = x$. This means exactly that 
$\tilde \Lambda_{x+1/2, f(x)} (e)$ and $\tilde \Lambda_{x-1/2, f(x)} (e)$ are equal to $f(e+1/2)$ or $f(e-1/2)$ depending on $e < x$ or $e >x$. 

But, for a given $f$ and $x$, the point $(x,f(x))$ is fixed and the forward curve $\tilde \Lambda_{x+1/2, f(x)}$ and the backward one $\tilde \Lambda_{x-1/2, f(x)}$ are distributed as independent random walks reflected at $0$ during the time-interval between $x$ and the origin and absorbed by $0$ outside this interval. It therefore follows that
\begin{align*}
 \PP(E_{x,f}) = \left(\frac{1}{2}\right)^{(m_+ - m_-) - O(f)}.
\end{align*}
We crucially see that $\PP (E_{x,f})$ is the same for all $x$ in $I(f)$.

\medbreak

Let us now suppose that $E_{x,f}$ holds, and study the value of $n$ and $k$ at which $(\tilde X_n, \tilde H_n )= (\tilde X_{N(k)}, \tilde H_{N(k)}) = (x,f(x))$.
Clearly both $n$ and $k$ are determined by $x$ and $f$. This is the content of the following Lemma:
\begin{lemma}
Let us fix $f \in {\mathcal C}$ and $x \in I (f)$. We denote by ${\mathcal A}(f)$ the area between $f$ and $\tilde a$ i.e. ${\mathcal A}(f) = \sum_{y} (f(y) - \tilde a(y))$. If $E_{x,f}$ holds, then the time $n(x,f) = N(k(x,f))$ at which $(\tilde X_n, \tilde H_n )= (x,f(x))$ verifies:
\begin{align*}
 n(x,f) = {\mathcal A}(f) + f(x).
\end{align*}
Consequently,
\begin{align}
 k(x,f) =  (n(x,f) + f(x))/2 = {\mathcal A}(f)/2 + f(x). \label{relkn}
\end{align}
\end{lemma}
\begin{proof}
 The time
$n=n(x,f)$ is equal to the area between the non modified local time associated to the pair $(x,f)$ and the initial local time $a$. Let us denote here the non modified function by $g$. For all $e \in E$ such that $e \le x-1/2$, $g(e) := f(e+1/2)$ and when $e \ge x-1/2$, $g(e) := f(e-1/2)$. Therefore, the time $n$ is equal to ${\mathcal A}(f)$, to which one has to add $f(x)$, because of the 
additional column of height $f(x)$ that one has removed in order to obtain $\tilde \ell$ out of $\ell$ (see Figure \ref{modifiedloctime}).
More precisely,
\begin{align*}
 n &= \sum_{e \in E} g(e) - a(e) \\
&=\sum_{\substack{e \in E\\ e \in [m_- -\frac12, x-\frac12]}} f(e+\frac12) + \sum_{\substack{e \in E\\ e \in [x +\frac12, m_+ + \frac12]}} f(e-\frac12)
- \sum_{\substack{e \in E \\e \in [m_- -\frac12,-\frac12]}} \tilde a(e+\frac12) -
\sum_{\substack{e \in E\\ e \in [\frac12, m_+ + \frac12]}} \tilde a(e-\frac12)  \\
&= f(x) - \tilde a(0) + \sum_{y \in \Z} f(y) - \tilde a(y) \\
&= f(x) + {\mathcal A}(f).
\end{align*}
The first equality in \eqref{relkn} follows from (\ref {N(k)}).
\end{proof}

In particular, we see that for a given $f$, the time $n(x,f)$ depends on the position of $x$ in $I(f)$. However, we will see that this dependency is mild: 
because the slope of $f$ is never larger than one, it follows that for any $x$, the area underneath $f$ is at least $|f(x)|^2$. With equality \eqref{relkn}, this implies that
\begin {equation}
\frac {{\mathcal A}(f)} 2  \le k(x,f) \le \frac {{\mathcal A}(f)} 2 + \sqrt {{\mathcal A}(f)}. \label{ineqpourk}
\end {equation}

\paragraph{Randomization of the observation.} For each $A$, let us define a geometric random variable $q_A$ with mean $A/2$. We want to describe the joint distribution of 
$$ ( x_A, \gamma_A ( \cdot) ) := (\tilde X_{N(q_A)}, \tilde \ell_{N(q_A)} ( \cdot)).$$
Suppose we only observe $\gamma_A (\cdot)$. As we have already indicated, the point $x_A$ is necessarily in the interval $I(\gamma_A)$. 
Let us sample uniformly an integer $u_{A}$ in this interval.  We now show that the law of $(u_A, \gamma_A)$ is close to that of $(x_A, \gamma_A)$ when $A$ is large.
More precisely:

\begin{lemma}\label{conddistribdiscrete}
The total variation distance between the distributions of $(u_A, \gamma_A)$ and $(x_A, \gamma_A)$ tends to $0$ as $A \to \infty$.
\end{lemma}

\begin{proof}

In order to prove the Lemma, we have to see that the probabilities that 
 $x_A = x$ and $\gamma_A = f$ are almost the same for all $x \in I(f)$. The reason for this result is that $f(x)$ will tend to be much smaller than ${\mathcal A}(f)$ in the limit when $A$ tends to $\infty$ for a proportion of $f$'s that goes to $1$. 
Clearly, this probability (for a given $f \in \mathcal{C}$ and $x \in I(f)$) is equal to 
$$ \PP\left( E_{x, f},\; q_A= k(x,f)\right)= \PP(E_{x,f}) \PP(q_A = k (x,f)) 
= \left(\frac{1}{2}\right)^{(m_+-m_-) - O(f)} \left(1-\frac{2}{A}\right)^{k(x,f)-1} \frac{2}{A}.$$

When $A \to \infty$, the time $q_A$ will typically be large so that $\sqrt{{\mathcal A}(\gamma_A)}$ will become negligible compared to ${\mathcal A}(\gamma_A)$. 
For every $f \in \mathcal{C}$, let us define
$$
p(f) = \left(\frac{1}{2}\right)^{(m_+-m_-) - O(f)} \left(1-\frac{2}{A}\right)^{{\mathcal A}(f)/2-1} \frac{2}{A}.
$$
Thanks to (\ref{ineqpourk}), for every $f \in \mathcal{C}$ and $x \in I(f)$, we can write:
$$
\left(1-\frac{2}{A}\right)^{\sqrt{\mathcal{A}(f)}} \le \frac { \PP ( (x_A, \gamma_A ) = (x,f))}{p(f)} \le 1.
$$
Taking the mean over $x$ in $I(f)$, we get that 
$$
\left(1-\frac{2}{A}\right)^{\sqrt{\mathcal{A}(f)}} \le \frac { \PP ( \gamma_A = f )}{p(f) \times \# I(f) } = \frac { \PP ( (u_A, \gamma_A ) = (x,f))}{p(f)} \le 1.
$$
Hence, 
$$
\left(1-\frac{2}{A}\right)^{\sqrt{\mathcal{A}(f)}} \le 
 \frac { \PP ( (x_A, \gamma_A ) = (x,f))} { \PP ( (u_A, \gamma_A ) = (x,f))} \le 
 \left(1-\frac{2}{A}\right)^{-\sqrt{\mathcal{A}(f)}} .
$$
It therefore suffices to prove that 
$$ 
\E \left( \left(1-\frac{2}{A}\right)^{\sqrt{\mathcal{A}(\gamma_A)}} \right) \to 1 
\hbox { and } \E \left( \left(1-\frac{2}{A}\right)^{-\sqrt{\mathcal{A}(\gamma_A)}} \right) \to 1 
$$ 
as $A \to \infty$. This is straightforward because  
$$
\frac {{\mathcal A}(\gamma_A)} 2  \le q_A  \le \frac {{\mathcal A}(\gamma_A)} 2 + \sqrt {{\mathcal A}(\gamma_A)}$$
keeping in mind also that $q_A$ is a geometric random variable with mean $A/2$. 
\end {proof}

\section{Convergence from discrete to continuous}\label{sec:CV}

The definition of the continuous TSRM by T\'oth and Werner uses a continuous analogue of the definition of $\tilde X_n$ via the maze that $(\tilde X_n, \tilde H_n)$ has to go through. One starts with a family of coalescing Brownian motions (instead of coalescing random walks) $(\Lambda_{x,h},\, (x,h)\in \R \times \R^*_+)$ starting from all points in the upper half-plane (such families had been constructed
by Arratia in \cite{Arratia}, and were further studied in \cite{TothWerner, STW, NewmanBW, NewmanCV} and are called Brownian web (BW) in the latter papers) which is shown to define a continuous plane-filling maze. For each $(x,h)$ in the upper half plane, $\Lambda_{x,h}$ has the distribution of a (one dimensional) two-sided Brownian motion with initial condition $\Lambda_{x,h}(x) = h$ reflected above $0$ in the time-interval between $0$ and $x$ and absorbed by $0$ outside this interval. One of the main property of the BW family is that almost surely its curves do not cross. The interaction between the BW-curves can be understood as follows:
One can imagine we first define the BW-curves starting from points belonging to a countable dense subset of the upper half plane $Q = (x_i,h_i)_{i \in \N^*}$. We also take $(x_0,h_0) := (0,0)$ and let $\Lambda_{x_0,h_0}$ be the function identically equal to $0$. We then construct the curves recursively. Given $(\Lambda_{x_i,h_i}, 0\le i < k)$, the \emph{forward} process $(\Lambda_{x_k,h_k}(x),x \ge x_k)$ has the law of an independent Brownian motion starting at $h_k$ at time $x_k$ coalescing with the \emph{forward} previously drawn curves $(\Lambda_{x_i,h_i}(x),x \ge x_i ; \, 0\le i < k)$ and reflected at the \emph{backward} curves $(\Lambda_{x_i,h_i}(x),x \le x_i ; \, i < k)$. The \emph{backward} process $(\Lambda_{x_k,h_k}(x),x \le x_k)$ is then an independent Brownian motion coalescing with the \emph{backward} previously drawn curves and reflected at the \emph{forward} ones. In this way, we defined the ``skeleton'' of the BW. One can extend this definition to the whole half plane by continuity (there is some freedom left for the construction of the other curves, we do not enter into the details here).

\medskip

The intuitive link between TSRM and BW goes as follows: Let us consider the process
$(X_t,H_t)$ started at $(0, 0)$ which traces the contour of the ``tree'' of these coalescing Brownian motions, then 
one obtains a space-filling curve that can be parametrized by the area it has swept. It means that for every $(x,h)$ in the upper half plane, the process $(X_t,H_t)$ visits the point $(x,h)$ at the random time $t = T_{x,h} := \int \Lambda_{x,h}$. There is a difficulty here because the set of times given by the family $(T_{x,h},\,(x,h) \in \R \times \R^*_+)$ is not the entire set of times $\R_+$. Nevertheless, it has good enough properties to enable to define the process $(X,H)$. 
 Its first coordinate $X$ will be the TSRM, while the second coordinate $H_t$ will turn out to be equal to its occupation time density (defined in the introduction) at its current position i.e. $ H_t = \Lambda_t ( X_t)$.  Moreover, the distribution (and even the joint distribution) of the occupation time density is known at the random times $T_{x,h}$. By construction, we have: a.s., for all $(x,h) \in \R \times \R^*_+$, $\Lambda_{T_{x,h}} = \Lambda_{x,h}$. We refer to \cite{TothWerner} for more details.

\medskip

Newman and Ravishankar \cite {NewmanCV} have shown that the (properly renormalized version of the) process $(\tilde X_n , \tilde H_n)$ converges to the continuous process $(X_t, H_t)$. We could try to extend it in order to also have convergence of the corresponding intervals $I$, but as we only need the convergence at the independent random time $\tau$, we will follow a more direct method. Let $\R$ and $\R^2$ be equipped with the Euclidean topology and let us denote by $C$ the set of continuous functions with compact support from $\R$ to $\R^+$ admitting a left-most and a right-most excursion such that $0$ belongs to the right-most or the left-most excursion, equipped with the uniform topology. 

\medskip

\noindent The goal is to establish the following convergence (where an interval $(x_-, x_+)$ in $\R$ is identified with a point $(x_-, x_+) \in \R^2$):

\begin{propo}\label{CVsimple}
The triplet
\begin{align}{\label{tripletsimple}}
(A^{-2/3}x_A, A^{-1/3} \gamma_A(A^{2/3}\cdot), A^{-2/3} I(\gamma_A)) = \left(A^{-2/3}\tilde X_{N(q_A)},\,A^{-1/3} \tilde \ell_{N(q_A)}(A^{2/3}\cdot),\, A^{-2/3} I\left(\ell_{N(q_A)}\right)\right)
\end{align}
converges in distribution towards
$$
(X_{\tau}, \; \Lambda_{\tau}(\cdot),\; I_{\tau}) \qquad \mbox{as } A \to \infty.
$$
(recall that $I_{\tau}$ is  the opposite excursion from the $0$-excursion in the continuous process $\Lambda_{\tau}$). 
\end{propo}

\begin{proof}
The trick will be to express the expectation of continuous bounded functionals of the triplet in terms of just one simple random walk/Brownian motion thanks to a simple change of variables from $\R_+$ into the upper half-space using the time-parametrization of the TSRM $t=T_{x,h}$. Indeed, the law of the occupation time is very easy to describe at the (random) first hitting times of $(x,h)$ both in the discrete and continuous models: In the discrete model, for $(x,h) \in \Z \times \N$ with $x+h$ even, it is a two-sided simple random walk properly reflected/absorbed and starting at $(x+1/2,h)$ for the forward one, $(x-1/2,h)$ for the backward one; in the continuous model, we have just seen that $\Lambda_{T_{x,h}} = \Lambda_{x,h}$.

\medskip

Let us take a continuous and bounded function $\varphi: \R \times C \times \R^2 \to \R$.  For each positive $A$, we define the rescaled functional $\varphi^A$ as
$$ \varphi^A ( x, \ell(\cdot), I  ) := \varphi (A^{-2/3} x, A^{-1/3} \ell ( A^{2/3} \cdot ), A^{-2/3} I).$$
 
For every $(x,h) \in \Z\times\N$ such that $x+h$ is even (we call $\mathcal F$ this set of pairs), we define $f_{x,h} \in \mathcal{C}$ to be the random continuous polygonal curve which concatenates the forward discrete BW-curve starting from $(x+1/2,h)$ and the backward one from $(x-1/2,h)$.
In other words, for every $(x,h) \in \mathcal {F}$ and $f \in \mathcal{C}$ such that $x \in I(f)$ and $f(x) = h$, we have $f_{x,h} := f$ on the integers if $E_{x,f}$ holds (and $f_{x,h}$ is then naturally extended to a polygonal curve).

Using (\ref{relkn}), we get:
\begin{align*}
&\E\left(\varphi^A \left( \tilde X_{N(q_A)},  \tilde \ell_{N(q_A)}(\cdot),  I (\ell_{N(q_A)} ) \right)\right)  \\
&= \sum_{k=0}^{\infty} \sum_{(x,h) \in {\mathcal F}}  \E\left(\varphi^A (x ,  f_{x,h}(\cdot), I(f_{x,h})) 1_{\{\mathcal{A}(f_{x,h}) + 2h = 2k\}} \right)  \PP(q_A = k)\\
&= \frac{2}{A} \sum_{(x,h) \in {\mathcal F}} \E \left( \varphi^A ( x, f_{x,h}(\cdot), I(f_{x,h})) \times  (1-2/A)^{(\mathcal{A}(f_{x,h}) + 2h )/2 -1} \right) 
\end {align*}

For all $(x,h)$ in the \emph{upper half-plane} $\R \times \R_+$, we now define the point $(x^A, h^A) \in {\mathcal F}$ such that $A^{2/3} x \in [x^A, x^A + 1)$  and $A^{1/3} h \in [h^A, h^A +1)$ or $ [h^A-1, h^A)$ depending on whether $x_A$ is even or odd, and define $f^A$ to be the rescaled function $f_{x^A, h^A}$ i.e. 
$$ f^A ( \cdot) := A^{-1/3} f_{x^A, h^A} ( A^{2/3} \cdot )$$
and we let ${\mathcal A}^A (f^A)$ denote the area between $f^A$ and the rescaled bottom line. Then, we can rewrite this last expression as
\begin {equation}
\label {funcdiscrete}
\int_{-\infty}^{\infty} \int_{0}^{\infty} \E \left( \varphi (A^{-2/3} x^A  , f^A (\cdot), I(f^A)) \times  (1-2/A)^{(A \, \mathcal{A}^A(f^A)  + 2 h^A)/2-1} \right) dx \ dh. 
\end {equation}
 
In the continuous setting, recall the definition of $T_{x,h} := \int \Lambda_{x,h}$ where $\Lambda_{x,h}$ is the BW-curve starting at $\Lambda_{x,h}(x) =h$. Using the (measure-preserving) change of variables $t = T_{x, h}$ (see Proposition 4.1 in \cite{TothWerner}) and Fubini's Theorem,
\begin{align}
 \E(\varphi(X_{\tau},\Lambda_{\tau},I_{\tau})) &= \E\left(\int_0^{\infty} e^{-t} \varphi(X_t,\Lambda_{X_t,\Lambda_t(X_t)},I_{X_t,\Lambda_t(X_t)})dt\right) \notag \\
 &= \E\left(\int_{-\infty}^{+\infty} \int_0^{\infty} e^{-T_{x,h}} \varphi(x,\Lambda_{x,h},I_{x,h})dh dx\right) \notag\\
&= \int_{-\infty}^{+\infty} \int_0^{\infty} \E\left(e^{-T_{x,h}} \varphi(x,\Lambda_{x,h},I_{x,h})\right)dh dx. \label{expressionC0}
\end{align}
where $I_{x,h}$ is the excursion of $\Lambda_{x,h}$ furthest from the origin.

Now, the convergence of (\ref{funcdiscrete}) towards this last expression boils down to a simple random walk/Brownian motion matter. Indeed, 
$f_{x^A,h^A}$ is the concatenation of two simple random walks and $\Lambda_{x,h}$ is a two-sided Brownian motion. A little care is needed here as $\Lambda_{x,h}$ is {not} a continuity point of the function which associates to $f \in C$ the opposite excursion from the $0$-excursion (see Billigsley \cite{Billingsley} for background). Nevertheless, the convergence holds thanks to the following classical argument: 

Let $h>0$ and let (for each $n$), $(S^n_k)_{k \in \N}$ denote a simple random walk starting at $[n^{1/3}h]$ and define
$S^{(n)}(\cdot):=n^{-1/3}S^n_{[n^{2/3}\cdot]}$ its renormalization. By Skorohod's representation Theorem, it is possible to couple all these walks $S^{(n)}$ with a one dimensional Brownian motion $B$ started at $h$ in such a way that $S^{(n)}$ almost-surely converges towards $B$ for the uniform topology on any compact time-interval. With continuity of $B$ and the fact that the Brownian motion almost surely becomes negative immediately after its first hitting time of the origin, it follows that the first hitting time of the $x$-axis by $S^{(n)}$ converges also almost-surely towards the first hitting time of the $x$-axis by $B$.

It follows that for every $(x,h) \in \R \times \R^*_+$,
\begin{align*}
\E \left( \varphi (A^{-2/3} x^A, f^A (\cdot), I(f^A)) \times  (1-2/A)^{(A \, \mathcal{A}^A(f^A)  + 2 h^A)/2-1} \right)
\longrightarrow \E\left(e^{-T_{x,h}} \varphi(x,\Lambda_{x,h},I_{x,h})\right)\end {align*}
as $A \to \infty$.
In order to deduce the convergence of \eqref{funcdiscrete} to \eqref{expressionC0}, it remains to apply the dominated convergence theorem (it is easy to see that the expectation (\ref{funcdiscrete}) admits the rough upper bound $\|\varphi\|_{\infty} \exp(-c ([x] + [h]))$ for all $A$ large enough and for some constant $c >0$ and all large $x$ and $h$ thanks to Markov property applied $[x]$ and $[h]$ times); this concludes the proof of  Proposition \ref{CVsimple}. 
\end{proof}

Combining Lemma \ref{conddistribdiscrete}. and Proposition \ref{CVsimple}. now also concludes the proof of our Theorem \ref{main2}.

\section{Probability to be perfectly hidden}\label{sec:probawholeint}

The burglar is therefore well hidden in the interval $I$. It can however happen that the interval $I$ is much shorter than the whole support of the local time (in particular when 
the burglar is close to its past maximum). It is natural to wonder what is the probability that $I_1$ is equal to the whole support $\mbox {supp}(\Lambda_1)$ (when this is the case, the burglar's strategy turned out to be particularly efficient...). This question is answered by the following statement:

\begin{propo}
The probability that $I_1$ is equal to the entire support of the local time is equal to 
$$1 - 9 \sqrt{3} \,\Gamma (2/3)^6 / (4 \pi^3) \approx 0.225.$$
\end{propo}

\begin {proof}
 
 Let us notice at first that with the scaling property, we have:
$$ \PP(\mbox {supp} (\Lambda_1) = I_1 ) = \PP( \mbox {supp} (\Lambda_{\tau}) = I_\tau),
$$ 
where $\tau$ is an independent exponential variable with mean $1$, just as before.
Recall that for a point $(x,h)$ in the upper half-plane, the probability that $\Lambda_{x,h}$ has only one excursion corresponds to the event that the Brownian motion that starts from the point $(x,h)$ in the direction of $0$ does hit $0$ ``on the other side'' of $0$. 
Using the measure preserving change of variable $t = T_{x,h}$ from $\R_+$ onto the upper half-plane, 
\begin{align}
\PP(\mbox {supp} (\Lambda_\tau) = I_\tau )
&= \int_0^\infty e^{-t} \EE \left( \mbox{supp} (\Lambda_t) = I_t \right)  dt \notag  \\ 
&= \int_{0}^{\infty} \int_{-\infty}^{\infty}  \EE\left(e^{-T_{x,h}} 1_{\mbox {supp}(\Lambda_{x,h} ) = I_{x,h}} \right) dx dh \notag \\
&= \int_{0}^{\infty} \int_{-\infty}^{\infty} 
\EE_h\left(\exp\left(-\int_{\xi'}^{\xi} B_t dt\right)  1_{ \xi' < x < \xi} \right) dx dh \notag 
\end{align}
where $(B,\PP_h)$ is a two-sided Brownian motion started at level $h$ at time $0$, and $\xi := \inf\{t\ge 0,\; B_t =0\}$ and $\xi' := \sup\{t\le 0,\; B_t =0\}$. 
With Fubini's theorem, we can now swap the expectation and the integral with respect to $x$, this shows that
\begin {align}
\PP( \mbox {supp} (\Lambda_{\tau}) = I_\tau ) 
&= 
 \int_{0}^{\infty}  \EE_h\left( (\xi - \xi') \exp\left(-\int_{\xi'}^{\xi} B_t dt\right)  \right)  dh
\notag \\
&= 
2  \int_{0}^{\infty}  \EE_h \left( \xi \exp \left( -\int_{\xi'}^{\xi} B_t dt \right)  \right)  dh.
\label {firsttag}
\end {align}
Note that with the scaling property of Brownian motion ($\xi$ and $\xi'$ scale like $h^2$ and the integral like $h^3$), 
\begin{align*}
 \EE_h \left( \xi \exp \left( -\int_{\xi'}^{\xi} B_t dt \right)  \right) &= \EE_1 \left(h^2 \xi \exp \left( - h^3 \int_{\xi'}^{\xi} B_t dt \right)  \right)\\
&= \frac{d}{dh}\left[\EE_1\left(\frac{\xi}{3 \int_{\xi'}^{\xi} B_t dt} \exp \left( - h^3 \int_{\xi'}^{\xi} B_t dt\right)\right)\right]
\end{align*}
(we can exchange the expectation $\EE_1$ and the differentiation with respect to $h$ because the function $h\mapsto h^2 \xi \exp \left( - h^3 \int_{\xi'}^{\xi} B_t dt \right)$ is bounded from above by $(2/(3e))^{2/3}\xi/(\int_{\xi'}^{\xi} B_t dt)^{2/3}$ whose expectation is finite).
Therefore, the integral \eqref{firsttag} is equal to
$2/3 \times \EE_1\left(\xi / \int_{\xi'}^\xi B_t dt \right)$, which seems however difficult to compute directly.

\medskip

\noindent Let us compute \eqref{firsttag} with a different method. Independence between the two-sides of the Brownian motion shows that
\begin{align}
\EE_h\left( \xi  \,\exp\left(-\int_{\xi'}^{\xi} B_t dt\right)\right)= u(h)\,  v(h) \label{simpleMP}
\end{align}
where
$$u(h) := \EE_h\left(\exp\left(-\int_0^{\xi} B_t dt\right) \right)$$
and
$$
v(h) := \EE_h\left(\xi \, \exp\left(-\int_0^{\xi} B_t dt\right)\right).$$
It is not difficult to see (and has been used in several papers, see for instance formula 2.8.1 p.~167 in \cite{Borodin} or \cite{Areatillfirstpassagetime}) that the function $u$ solves the differential equation: 
\begin{align}\label{diffequ}
 u''(x) = 2 x\, u(x)
\end{align}
 with initial condition $u(0) = 1$ (and is bounded on $\R^+$) which implies that it can be expressed in terms of the airy function $\Ai$: $u(\cdot)= \Ai(2^{1/3} \cdot)/\Ai(0)$.

\noindent Let us now show that 
\begin {align}
v(h)
&= u(h) u'(0) - u'(h). \label{relEh}
\end{align}
We fix $h >0$, $\eps >0$ and define $\tilde \xi$ to be the first time at which a Brownian motion $B$ started at $h+\eps$ hits $\eps$. The strong Markov property 
shows that
\begin{align*}
 u(h + \eps) &= \EE_{h+\eps}\left(\exp\left(- \int_0^{\tilde \xi} B_t dt\right)\right) u(\eps)  \\
&= u(\eps)  \EE_{h}\left(\exp\left(- \int_0^{\xi} (B_t + \eps) dt\right)\right) \\
&= u(\eps) \EE_{h}\left(\exp ( - \eps \xi) \times \exp\left(- \int_0^{\xi} B_t dt\right)\right)
\end {align*}
so that 
\begin {align*}
u( h+ \eps) - u (h) &= 
u(\eps)  \EE_{h}\left((\exp ( - \eps \xi) -1)  \times \exp\left(- \int_0^{\xi} B_t dt\right)\right)
\\& \null \quad + (u(\eps) - 1)  \EE_{h}\left( \exp\left(- \int_0^{\xi} B_t dt\right)\right).
\end {align*}
Letting $\eps$ to $0$, using the fact that we know that $u$ is $C^1$, we get (\ref {relEh}) by bounded convergence. 

Notice also that with an integration by parts and the differential equation (\ref{diffequ}), we get that:
\begin{align}
\int_0^{\infty} u^2(y) dy = u'(0)^2/2 \label{intu2} 
\end{align}

\noindent We are now ready to conclude: The relations (\ref {firsttag}), (\ref{simpleMP}), (\ref{relEh}) and (\ref{intu2}) lead to:
\begin{align*}
  \PP(\mbox {supp}(\Lambda_1) = I_1 ) &=
2 \int_0^\infty dh u(h) v(h) \\ 
&= 2 \int_0^\infty u(h)^2 u'(0) dh - 2 \int_0^\infty u'(h) u(h) dh \\
&=  u'(0)^3 + 1
\end{align*}
Note that by 10.4.4 and 10.4.5 \cite{Abramowitz}, we have $u'(0) = -6^{1/3} \Gamma(2/3)/\Gamma(1/3)$ and the studied probability is equal to $ 1 - (6\, \Gamma(2/3)^3 / {\Gamma(1/3)^3}) = 1 - 9 \sqrt{3} \Gamma (2/3)^6 / (4 \pi^3)$.
\end {proof}

\medbreak

\noindent \textbf{Acknowledgements:} I warmly thank my supervisor Wendelin Werner for his precious help and advices throughout this work and for his careful reading of this paper. I am also grateful to the anonymous referee for his/her useful comments. 

\bibliographystyle{plain}
\bibliography{biblioloicond}
D\'epartement de Math\'ematiques et Applications

Ecole Normale Sup\'erieure 

75230 Paris cedex 05

France 

\medbreak

laure.dumaz@ens.fr

\end{document}